\documentclass[a4paper,10pt]{article}

\usepackage[utf8]{inputenc} 
\usepackage{amsmath}
\usepackage{amsbsy}
\usepackage{amssymb}
\usepackage{amscd,amsthm}
\usepackage{graphicx}
\usepackage[mathcal]{eucal}
\usepackage{verbatim}
\usepackage[english]{babel}
\usepackage[dvigs]{epsfig}
\usepackage{dsfont}
\usepackage{longtable}
\usepackage{marvosym}
\usepackage{anysize}

\newcommand{\Z}{\mathds{Z}}
\newcommand{\Q}{\mathds{Q}}
\newcommand{\N}{\mathds{N}}
\newcommand{\R}{\mathds{R}}

\newcommand{\p}{\phantom}
\newcommand{\q}{\quad}

\newtheorem{thm}{Theorem}[section]
\newtheorem{lem}[thm]{Lemma}
\newtheorem{kor}[thm]{Corollary}

\newtheorem{prop}[thm]{Proposition}

\theoremstyle{definition}
\newtheorem*{defi}{Definition}

\theoremstyle{remark}
\newtheorem*{rema}{Remark}

\title{On the sum of the first $n$ prime numbers} 
\author{Christian Axler}

\begin{document}

\maketitle

\begin{abstract}
In this paper we establish a general asymptotic formula for the sum of the first $n$ prime numbers, which leads to a generalization of the most accurate
asymptotic formula given by Massias and Robin in 1996.
\end{abstract}

\section{Introduction}
At the beginning of the 20th century, Landau \cite{el} showed that
\begin{displaymath}
\sum_{k \leq n} p_k \sim \frac{n^2}{2} \log n \q\q (n \to \infty).
\end{displaymath}
The currently most accurate asymptotic formula was proved by Massias and Robin \cite{mr} in 1996, namely
\begin{equation} \label{101}
\sum_{k \leq n} p_k = \frac{n^2}{2} \left( \log n + \log \log n - \frac{3}{2} + \frac{\log \log n - 5/2}{\log n} \right) + O \left( \frac{n^2(\log \log
n)^2}{\log^2 n} \right).
\end{equation}

\section{An asymptotic formula for $\sum_{k \leq n} p_k$}

In this paper we derive a general asymptotic formula for the sum of the first $n$ prime numbers. To prove this asymptotic formula, we introduce the
following definition.

\begin{defi}
Let $s,i,j,r \in \N_0$ with $j \geq r$. We define the integers $b_{s,i,j,r} \in \Z$ as follows:
\begin{itemize}
 \item If $j=r=0$, then
\begin{equation} \label{202}
b_{s,i,0,0} = 1.
\end{equation}
 \item If $j \geq 1$, then
\begin{equation} \label{203}
b_{s,i,j,j} = b_{s,i,j-1,j-1} \cdot (-(i-(j-1))).
\end{equation}
 \item If $j \geq 1$, then
\begin{equation} \label{204}
b_{s,i,j,0} = b_{s,i,j-1,0} \cdot (s+j-1).
\end{equation}
 \item If $j > r \geq 1$, then
\begin{equation} \label{205}
b_{s,i,j,r} = b_{s,i,j-1,r} \cdot (s+j-1) + b_{s,i,j-1,r-1} \cdot (-(i-(r-1))).
\end{equation}
\end{itemize}
\end{defi}

\begin{prop} \label{p201}
If $r > i$, then
\begin{displaymath}
b_{s,i,j,r} = 0.
\end{displaymath}
\end{prop}

\begin{proof}
Let $r>i$. From \eqref{203}, it follows that
\begin{equation} \label{206}
b_{s,i,i+1,i+1} = 0
\end{equation}
and hence
\begin{equation} \label{207}
b_{s,i,k,k} = 0
\end{equation}
for every $k \geq i+1$. We use \eqref{205} repeatedly and \eqref{206} to get
\begin{displaymath}
b_{s,i,k,i+1} = b_{s,i,i+1,i+1} \cdot (s+k-1) \cdot \ldots (s+(i+2)-1) = 0
\end{displaymath}
for every $k\geq i+2$. By using \eqref{206}, it follows that
\begin{equation} \label{208}
b_{s,i,k,i+1} = 0
\end{equation}
for every $k \geq i+1$. Next, we prove by induction that
\begin{equation} \label{209}
b_{s,i,k,i+n} = 0
\end{equation}
for every $n \in \N$ and every $k \geq i+n$. If $k = i+n$, then $b_{s,i,k,i+n} = 0$ by \eqref{207}. So, it suffices to prove \eqref{209} for every $n \in
\N$ and every $k \geq i+n+1$. If $n=1$, the claim follows from \eqref{208}. Now we write $k=i+n+t$ with an arbitrary $t \in \N$. By \eqref{205} and the
induction hypothesis, we obtain
\begin{displaymath}
b_{s,i,t+i+n,i+n} = b_{s,i,i+n,i+n} \cdot (s+(t+n+i)-1) \cdot \ldots \cdot (s+(i+n+1)-1).
\end{displaymath}
Since $b_{s,i,i+n,i+n} = 0$ by \eqref{207}, we get $b_{s,i,k,i+n} = 0$ and the result follows.
\end{proof}

\noindent
Let $m \in \N$. By \cite{cp}, there exist unique $a_{is} \in \Q$, where $a_{ss} = 1$ for all $1 \leq s \leq m$, such that
\begin{equation} \label{210}
p_n = n\left( \log n + \log \log n - 1 + \sum_{s=1}^m \frac{(-1)^{s+1}}{s\log^s n} \sum_{i=0}^s a_{is}(\log \log n)^i \right) + O (c_m(n)),
\end{equation}
where
\begin{displaymath}
c_m(x) = \frac{x(\log \log x)^{m+1}}{\log^{m+1} x}.
\end{displaymath}
We set
\begin{displaymath}
g(x) = \log x + \log \log x - \frac{3}{2}
\end{displaymath}
and
\begin{displaymath}
h_m(x) = \sum_{j=1}^{m} \frac{(j-1)!}{2^j \log^j x}.
\end{displaymath}

\noindent
The \emph{logarithmic integral} $\text{li}(x)$ is defined for every real $x \geq 2$ as
\begin{displaymath}
\text{li}(x) = \int_0^x \frac{dt}{\log t} = \lim_{\varepsilon \to 0} \left \{ \int_{0}^{1-\varepsilon}{\frac{dt}{\log t}} +
\int_{1+\varepsilon}^{x}{\frac{dt}{\log t}} \right \} \approx \int_2^x \frac{dt}{\log t} + 1.04516....
\end{displaymath}

\begin{lem} \label{l202}
We have
\begin{displaymath}
\emph{li}(x) = \sum_{j=1}^{n} \frac{(j-1)!}{\log^j x} + O \left( \frac{x}{\log^{n+1} x} \right).
\end{displaymath}
\end{lem}

\begin{proof}
Integration by parts.
\end{proof}

\begin{lem} \label{lem403}
Let $x,a \in \R$ be such that $x \geq a \geq 2$. Then
\begin{displaymath}
\int_a^x \frac{t\, dt}{\log t} = \emph{li}(x^2) - \emph{li}(a^2).
\end{displaymath}
\end{lem}

\begin{proof}
See Dusart \cite{pd}.
\end{proof}

\begin{lem} \label{lem404}
Let $m,m_0 \in \mathds{N}$ be such that $m \geq m_0$ and let $f$ be a continuous function on $[m_0, \infty)$ which is non-negativ and increasing on $[m,
\infty)$. Then
\begin{displaymath}
\sum_{k=m_0}^n f(k) = \int_{m_0}^n f(x) \, dx + O(f(n)).
\end{displaymath}
\end{lem}

\begin{proof}
We estimate the integral by upper and lower sums.
\end{proof}

\noindent
The following theorem yields a general asymptotic formula for the sum of the first $n$ primes.

\begin{thm} \label{thm402}
For every $n \in \N$,
\begin{displaymath}
\sum_{k \leq n} p_k = \frac{n^2}{2} \left( g(n) - h_m(n) + \sum_{s=1}^{m} \frac{(-1)^{s+1}}{s\log^sn} \sum_{i=0}^s a_{is} \sum_{j=0}^{m-s}
\sum_{r=0}^{\min\{i,j\}} \frac{b_{s,i,j,r}(\log \log n)^{i-r}}{2^j\log^jn} \right)+ O (nc_m(n)).
\end{displaymath}
\end{thm}

\begin{proof}
We set
\begin{displaymath}
\tau(x) = x\left( \log x + \log \log x - 1 + \sum_{s=1}^m \frac{(-1)^{s+1}}{s\log^s x} \sum_{i=0}^s a_{is}(\log \log x)^i \right).
\end{displaymath}
By \eqref{210}, we obtain
\begin{displaymath}
p_k = \tau(k) + O(c_m(k)).
\end{displaymath}
Using $\tau(n) \sim n\log n$ as $n \to \infty$ and Lemma \ref{lem404}, we get
\begin{equation} \label{211}
\sum_{k \leq n} p_k = \sum_{k=3}^n \tau(k) + O ( nc_m(n)) = \int_3^n \tau(x)\, dx + O(nc_m(n)).
\end{equation}
First, we integrate the first three terms of $\tau(x)$. We have
\begin{displaymath}
\int_3^n (x \log x - x) \, dx = \frac{n^2 \log n}{2} - \frac{3n^2}{4} + O(1).
\end{displaymath}
Next, using integration by parts, Lemma \ref{l202} and Lemma \ref{lem403}, we get
\begin{displaymath}
\int_3^n x \log \log x \, dx = \frac{n^2 \log \log n}{2} - \frac{n^2}{2}\, h_m(n) + O \left( \frac{n^2}{\log^{m+1} n} \right).
\end{displaymath}
Hence, by \eqref{211},
\begin{equation} \label{gl414}
\sum_{k \leq n} p_k = \frac{n^2}{2}(g(n) - h_m(n)) + \sum_{s=1}^m \frac{(-1)^{s+1}}{s} \sum_{i=0}^s a_{is} \int_3^n \frac{x(\log \log x)^i}{\log^s x}\, dx +
O(nc_m(n)).
\end{equation}
Now let $1 \leq s \leq m$ and $0 \leq i \leq s$. We prove by induction that for every $t \in \N_0$,
\begin{align} \label{gl415}
\int_3^n \frac{x(\log \log x)^i}{\log^s x}\, dx & = \sum_{j=0}^t \sum_{r=0}^{\min\{i,j\}} \frac{b_{s,i,j,r}n^2(\log \log n)^{i-r}}{2^{j+1}\log^{s+j}n} \nonumber \\
& \p{\q\q} + \int_3^n \sum_{r=0}^{\min\{i,t+1\}} \frac{b_{s,i,t+1,r}x(\log \log x)^{i-r}}{2^{t+1}\log^{s+t+1}x} \,dx + O(1).
\end{align}
Integration by parts gives
\begin{displaymath}
\int_3^n \frac{x(\log \log x)^i}{\log^s x}\, dx = \frac{n^2(\log \log n)^i}{2\log^s n} - \frac{i}{2} \int_3^n \frac{x(\log \log x)^{i-1}}{\log^{s+1} x}\, dx + \frac{s}{2} \int_3^n \frac{x(\log \log x)^i}{\log^{s+1} x}\, dx + O(1),
\end{displaymath}
so \eqref{gl415} holds for $t=0$. By induction hypothesis, we get
\begin{displaymath}
\int_3^n \frac{x(\log \log x)^i}{\log^s x}\, dx = \sum_{j=0}^{t-1} \sum_{r=0}^{\min\{i,j\}} \frac{b_{s,i,j,r}n^2(\log \log n)^{i-r}}{2^{j+1}\log^{s+j}n} +
\sum_{r=0}^{\min\{i,t\}} \int_3^n \frac{b_{s,i,t,r}x(\log \log x)^{i-r}}{2^t\log^{s+t}x} \,dx + O(1).
\end{displaymath}
Using integration by parts of the integral on the right hand side, we obtain
\begin{align} \label{gl416}
\int_3^n \frac{x(\log \log x)^i}{\log^s x}\, dx & = \sum_{j=0}^t \sum_{r=0}^{\min\{i,j\}} \frac{b_{s,i,j,r}n^2(\log \log n)^{i-r}}{2^{j+1}\log^{s+j}n}
\nonumber \\
& \p{\q\q} + \int_3^n \sum_{r=0}^{\min\{i,t\}} \frac{b_{s,i,t,r}((s+t)(\log \log x)^{i-r} - (i-r)(\log \log x)^{i-r-1})x}{2^{t+1}\log^{s+t+1}x} \,dx
\nonumber \\
& \p{\q\q} + O(1).
\end{align}
Since
\begin{align*}
\sum_{r=0}^{\min\{i,t\}} & b_{s,i,t,r}((s+t)(\log \log x)^{i-r} - (i-r)(\log \log x)^{i-r-1}) \nonumber\\
& \p{\q\q} = \sum_{r=1}^{\min\{i,t\}} ( b_{s,i,t,r}(s+t) - b_{s,i,t,r-1}(i-(r-1)))(\log \log x)^{i-r} + b_{s,i,t,0}(s+t)(\log \log x)^i \nonumber\\
& \p{\q\q\q\q\q\q} - b_{s,i,t,\min\{i,t\}}(i-\min\{i,t\})(\log \log x)^{i-(\min\{i,t\}+1)},
\end{align*}
we can use \eqref{204} and \eqref{205} to get
\begin{align} \label{gl417}
\sum_{r=0}^{\min\{i,t\}} & b_{s,i,t,r}((s+t)(\log \log x)^{i-r} - (i-r)(\log \log x)^{i-r-1}) \nonumber \\
& \p{\q\q} = \sum_{r=0}^{\min\{i,t\}} b_{s,i,t+1,r}(\log \log x)^{i-r} - b_{s,i,t,\min\{i,t\}}(i-\min\{i,t\})(\log \log x)^{i-(\min\{i,t\}+1)}.
\end{align}
It is easy to see that
\begin{displaymath}
- b_{s,i,t,\min\{i,t\}}(i-\min\{i,t\}) = b_{s,i,t+1,\min\{i,t\}+1}.
\end{displaymath}
Hence by \eqref{gl417} we obtain
\begin{displaymath}
\sum_{r=0}^{\min\{i,t\}} b_{s,i,t,r}((s+t)(\log \log x)^{i-r} - (i-r)(\log \log x)^{i-r-1}) = \sum_{r=0}^{\min\{i,t\}+1} b_{s,i,t+1,r}(\log \log x)^{i-r}.
\end{displaymath}
Since $b_{s,i,t+1,i+1} = 0$ for $t \geq i$ by Proposition \ref{p201}, it follows that
\begin{displaymath}
\sum_{r=0}^{\min\{i,t\}} b_{s,i,t,r}((s+t)(\log \log x)^{i-r} - (i-r)(\log \log x)^{i-r-1}) = \sum_{r=0}^{\min\{i,t+1\}} b_{s,i,t+1,r}(\log \log x)^{i-r}.
\end{displaymath}
Using \eqref{gl416}, we obtain \eqref{gl415}. Now we choose $t=m-s$ in \eqref{gl415} and we get
\begin{displaymath}
\int_3^n \frac{x(\log \log x)^i}{\log^s x}\, dx = \sum_{j=0}^{m-s} \sum_{r=0}^{\min\{i,j\}} \frac{b_{s,i,j,r}n^2(\log \log n)^{i-r}}{2^{j+1}\log^{s+j}n} +
O \left( \frac{n^2(\log \log n)^i}{\log^{m+1}n} \right).
\end{displaymath}
We substitute this in \eqref{gl414} to obtain
\begin{align*}
\sum_{k \leq n} p_k & = \frac{n^2}{2}(g(n) - h_m(n)) + \sum_{s=1}^m \frac{(-1)^{s+1}}{s} \sum_{i=0}^s a_{is} \sum_{j=0}^{m-s} \sum_{r=0}^{\min\{i,j\}}
\frac{b_{s,i,j,r}n^2(\log \log n)^{i-r}}{2^{j+1}\log^{s+j}n} \\
& \p{\q\q} + O \left( \sum_{s=1}^m \frac{(-1)^{s+1}}{s} \sum_{i=0}^s a_{is} \frac{n^2(\log \log n)^i}{\log^{m+1}n} \right) + O(nc_m(n)),
\end{align*}
and our theorem is proved.
\end{proof}

\noindent
The following corollary generalizes the asymptotic formula \eqref{101}.

\begin{kor} \label{k205}
Let $m \in \N$. Then there exist unique monic polynomials $T_s \in \Q[x]$, where $1 \leq s \leq m$ and $\emph{deg}(T_s) = s$, such that
\begin{displaymath}
\sum_{k \leq n} p_k = \frac{n^2}{2} \left( \log n + \log \log n - \frac{3}{2} + \sum_{s=1}^m \frac{(-1)^{s+1}T_s(\log \log n)}{s\log^s n} \right) +
O(nc_m(n)).
\end{displaymath}
The polynomials $T_s$ can be computed explicitly. In particular, $T_1(x) = x - 5/2$ and $T_2(x) = x^2 - 7x + 29/2$.
\end{kor}

\begin{proof}
Since $a_{ss} = 1$ and $b_{s,s,0,0} = 1$, the first claim follows from Theorem \ref{thm402}. By using \cite{cp} and setting $m=2$ in Theorem \ref{thm402},
we obtain the polynomials $T_1$ and $T_2$.
\end{proof}

\begin{rema}
The first part of Corollary \ref{k205} was already proved by Sinha \cite{sin} in 2010. Due to a calculation error, he gave the polynomials $T_1(x) = x -
3$ and $T_2(x) = x^2 - 7x + 27/2$.
\end{rema}


\end{document}